\documentclass[11pt,a4paper]{article}
\usepackage{amssymb}
\usepackage{eurosym}
\usepackage{amsfonts}
\usepackage{amsmath}
\usepackage{amsthm}
\usepackage{graphicx}
\usepackage{float}
\usepackage{hyperref}
\usepackage[pagewise]{lineno}

\hypersetup{
	colorlinks=true,
	linkcolor=blue,
	anchorcolor=blue,
	citecolor=blue
}
\newtheorem{theorem}{Theorem}[section]

\newtheorem{conjecture}[theorem]{Conjecture}

\newtheorem{lemma}[theorem]{Lemma}
\newtheorem{proposition}[theorem]{Proposition}
\theoremstyle{definition}
\newtheorem{definition}[theorem]{Definition}
\newtheorem{example}[theorem]{Example}

\newtheorem{remark}[theorem]{Remark}

\newtheorem*{da}{Data availability}
\renewenvironment{proof}[1][Proof]{\noindent\textbf{#1.} }{\ \rule{0.5em}{0.5em}}
\newenvironment{acknowledgement}{\smallskip{\sc Acknowledgement.}\rm}{\smallskip}
\renewcommand{\theequation}{\thesection.\arabic{equation}}
\allowdisplaybreaks
\input{tcilatex}

\def\func#1{\mathop{\mathrm{#1}}\nolimits}

\def\Xint#1{\mathchoice
{\XXint\displaystyle\textstyle{#1}}%
{\XXint\textstyle\scriptstyle{#1}}%
{\XXint\scriptstyle\scriptscriptstyle{#1}}%
{\XXint\scriptscriptstyle\scriptscriptstyle{#1}}%
\!\int}
\def\XXint#1#2#3{{\setbox0=\hbox{$#1{#2#3}{\int}$ }
\vcenter{\hbox{$#2#3$ }}\kern-.6\wd0}}

\def\oint{\Xint-}

\def\enddoc{\end{document}}

\def\FRAME#1#2#3#4#5#6#7#8
{
 \begin{figure}[H]
 \begin{center}
 \includegraphics[height=#3]{#7}
 \caption{#5}
 \label{#6}
 \end{center}
 \end{figure}
}

\begin{document}
	\title{Finite extinction time for subsolutions of the weighted Leibenson equation on Riemannian manifolds}
	\author{Philipp S\"urig}
	\date{January 2026}
	\maketitle
	
	\begin{abstract}
		We consider on Riemannian manifolds the non-linear evolution equation $$\rho \partial _{t}u=\Delta _{p}u^{q}.$$ Assuming that the manifold satisfies a \textit{(weighted) Sobolev inequality} and under certain assumptions on $p, q$ and function $\rho$, we prove that weak subsolutions to this equation have a finite extinction time. In particular, our main result holds in the case of a \textit{Cartan-Hadamard manifold}.
		
	\end{abstract}
	
	\let\thefootnote\relax\footnotetext{\textit{\hskip-0.6truecm 2020 Mathematics Subject Classification.} 35K55, 58J35, 35K92. \newline
		\textit{Key words and phrases.} Leibenson equation, doubly nonlinear
		parabolic equation, Riemannian manifold, finite extinction time. \newline
		The author was funded by the Deutsche Forschungsgemeinschaft (DFG,
		German Research Foundation) - Project-ID 317210226 - SFB 1283.}
	
	\tableofcontents
	
	\section{Introduction}

	Let $M$ be a Riemannian manifold. We consider solutions of the non-linear evolution
	equation 
	\begin{equation}
		\rho\partial _{t}u=\Delta _{p}u^{q},  \label{evoeq}
	\end{equation}%
	where $p>1$, $q>0$, $\rho=\rho(x)$ is a positive function of $x\in M$, $u=u(x,t)$ is an unknown non-negative function of $x\in M$, $t\geq0$ and $%
	\Delta _{p}$ is the Riemannian $p$-Laplacian 
	\begin{equation*}
		\Delta _{p}v=\func{div}\left( |\nabla v|^{p-2}\nabla v\right) .
	\end{equation*}

When $\rho\equiv1$, equation (\ref{evoeq}) becomes the \textit{Leibenson equation} \begin{equation}\label{usualleibe}\partial_{t}u= \Delta _{p}u^{q},\end{equation} which is also referred to as a \textit{doubly non-linear parabolic equation}.

In the present paper, we are interested in the existence of some \textit{finite extinction time} $T>0$ for solutions of equation (\ref{evoeq}), which means that, \begin{equation}\label{finiteextinct}u(\cdot, t)\equiv0\quad \text{for all}~t\geq T.\end{equation}

Assume that \begin{equation}\label{q(p-1)int}D:=1-q(p-1)>0.\end{equation}
Let us cite a known example of finite extinction time for (\ref{evoeq}) when $M=\mathbb{R}^{n}$. Let function $\rho(x)$ satisfy, for some $l\geq 0$, \begin{equation}\label{rhornl}\rho(x)=|x|^{-l}, \quad |x|\geq 1.\end{equation} Then by \cite{tedeev2007interface}, if $l$ satisfies  \begin{equation}\label{last}l^{\ast}:=\frac{p-nD}{1-D}<l<p<n,\end{equation} then (\ref{evoeq}) admits a solution with finite extinction time, namely 
\begin{equation}\label{pseudobarenblatt}u(x, t)=(T-t)_{+}^{\frac{n-l}{\kappa_{l}}}\left[C+\kappa_{l}^{\frac{1}{p-1}}\frac{D}{(p-l)q}|x|^{\frac{p-l}{p-1}}(T-t)_{+}^{\frac{p-l}{(p-1)\kappa_{l}}}\right]^{-\frac{p-1}{D}},\end{equation} where $C>0$ is any constant and $$\kappa_{l}=(1-D)(l-l^{\ast})$$ (cf. \cite{tedeev2007interface}).

In particular, if \begin{equation}\label{condunweight}p< nD,\end{equation} we can choose $l=0$ in (\ref{last}), in which case function $u$ from (\ref{pseudobarenblatt}) solves the unweighted Leibenson equation (\ref{usualleibe}). When we also assume $p=2$, equation (\ref{usualleibe}) becomes the \textit{parabolic porous medium equation} and condition (\ref{condunweight}) amounts to $q<\frac{n-2}{n}$. In this case, finite extinction time for solutions of the Leibenson equation (\ref{usualleibe}) was, for example, investigated in \cite{benilan1981continuous, blanchet2009asymptotics, DaskalopoulosSesum+2008+95+119, Winkler2011}.

In the present paper, we prove finite extinction time for solutions of equation (\ref{evoeq}) on Riemannian manifolds under certain assumptions on the manifold $M$ and function $\rho$.

Given a positive smooth function $h$, let $\mu$ be the measure defined by $$d\mu=hd\textnormal{vol},$$ where $\textnormal{vol}$ denotes the \textit{Riemannian measure} on $M$. The pair $(M, \mu)$ is called \textit{weighted manifold}. We assume that the weighted manifold $(M, \mu)$ is geodesically complete.

The first main result of the present paper is as follows (cf. \textbf{Theorem \ref{thmfinex}}).

\begin{theorem}\label{thmfinexint}
Assume that (\ref{q(p-1)int}) holds. Suppose that the weighted manifold $(M, \mu)$ satisfies the Sobolev inequality \begin{equation}\label{Sobolevinequality}\left(\int_{M}{|f|^{p\kappa} d\mu}\right)^{1/\kappa}\leq C\int_{M}{|\nabla f|^{p}d\mu}\quad \forall~f\in W^{1, p}(M),\end{equation}  with some Sobolev exponent $\kappa>1$. Let $u$ be a non-negative bounded solution of (\ref{evoeq}) in $M\times [0, \infty)$, with initial function $u(\cdot, 0)=u_{0}\in L^{\zeta}(M, \rho d\mu)\cap L^{\infty}(M)$ for some $\zeta\geq1$. If $\rho$ is bounded and satisfies \begin{equation}\label{finitecondint}\left|\left|\rho\right|\right|_{L^{\theta}(M, \mu)}<\infty,\end{equation} for some $\theta=\theta(p, q, \kappa, \zeta)$, then $u$ has a finite extinction time in $M$.
\end{theorem}

Let us emphasize that the only assumption that we impose on the manifold, besides the completeness, is the Sobolev inequality (\ref{Sobolevinequality}). For example, this inequality holds if $(M, \mu)$ is a \textit{Cartan-Hadamard manifold} \cite{hoffman1974sobolev}, that is, $M$ is simply connected and has everywhere non-positive sectional curvature; in this case $\kappa=\frac{n}{n-p}$, where $n=\textnormal{dim}~M>p$.

In fact, we prove a more general result Theorem \ref{thmfinex}, where the weighted manifold $(M, \mu)$ satisfies a certain weighted Sobolev inequality (see Section \ref{Secini}), which contains Theorem \ref{thmfinexint}. 

A finite extinction time for solutions of (\ref{evoeq}) in the case (\ref{q(p-1)int}) on a Riemannian manifold was obtained in \cite{andreucci2021extinction} by D. Andreucci and A. F. Tedeev under the hypotheses of a certain isoperimetric inequality and upper bounds of the volume growth function. These conditions imply a weighted Sobolev inequality, so that the result in \cite{andreucci2021extinction} is covered by our result Theorem \ref{thmfinex}. 

When $p=2$ and $\rho\equiv 1$, that is, (\ref{evoeq}) becomes the \textit{porous medium equation} and $D=1-q$, finite extinction time was proved in \cite{bonforte2008fast} by M. Bonforte, G. Grillo, and J. L. Vazquez when $M$ has non-positive curvature and $q<\frac{n-2}{n}$ or $M$ has curvature $k<0$ and $\frac{n-2}{n}<q<1$.

In fact, we prove Theorem \ref{thmfinexint} for  \begin{equation}\label{defthetaint}\theta=\frac{\kappa}{ \kappa-1-\frac{D}{\sigma}}\end{equation} where \begin{equation}\label{sigmalargeintreal}\sigma= \max\left(pq, \zeta-D, \frac{D}{\kappa-1}\right).\end{equation}
Note that, $\theta$ could be finite or infinite but we always have $\theta>\frac{\kappa}{\kappa-1}$. Therefore, finite extinction time occurs also if \begin{equation}\label{weakcondrho}\left|\left|\rho\right|\right|_{L^{\frac{\kappa}{\kappa-1}}(M, \mu)}<\infty.\end{equation}
For example, in the case of a Cartan-Hadamard manifold, (\ref{weakcondrho}) becomes $$\left|\left|\rho\right|\right|_{L^{\frac{n}{p}}(M, \mu)}<\infty .$$

The restriction $\sigma\geq pq$ in (\ref{sigmalargeintreal}) seems to be technical. It comes from the proof of the Caccioppoli inequality from \cite{grigor2023finite} (see Section 2.2 therein). If one can prove this inequality without this restriction, then one obtains the following:

\begin{conjecture}\label{conjintr}
The statement of Theorem \ref{thmfinexint} holds with \begin{equation}\label{thetamax}\theta=\left\{\begin{array}{ll}
		\frac{\kappa}{\kappa-1-\frac{D}{\zeta-D}}, & \text{if }\kappa\geq \frac{1}{\zeta-D}, \\ 
		\infty, & \text{if }\kappa< \frac{1}{\zeta-D},%
	\end{array}%
	\right.\end{equation} that is, $$\sigma=\max\left(\zeta-D, \frac{D}{\kappa-1}\right).$$ 
\end{conjecture}
Clearly, $\theta$ in (\ref{defthetaint}) is strictly smaller than $\theta$ in (\ref{thetamax}) in the case $\kappa\geq\frac{1}{\zeta-D}$.

If Conjecture \ref{conjintr} is true and $\zeta=1$, then it gives for $M=\mathbb{R}^{n}$ the following. The condition $\kappa< \frac{1}{1-D}$ is in this case equivalent to $$n>\frac{p}{D},$$ that is, to (\ref{condunweight}). Under this condition we have $\theta=\infty$, so that (\ref{finitecondint}) holds for $\rho\equiv1$. Hence, we conclude finite extinction time for solutions of the unweighted Leibenson equation (\ref{usualleibe}) in $\mathbb{R}^{n}$. However, this conjecture is not yet proved and Theorem \ref{thmfinexint} gives the following sufficient condition for finite extinction time for solutions of (\ref{usualleibe}): $$\sigma=\frac{D}{\kappa-1}=\frac{D(n-p)}{p},$$ which is satisfied if $n$ is large enough as can be seen from (\ref{sigmalargeintreal}). The example of $\mathbb{R}^{n}$ is discussed in Section \ref{examplessec}. Moreover, we construct in Proposition \ref{weightmodrhoone} (see also Lemma \ref{weightedmodelrho}) a class of \textit{weighted models} $(M, \mu)$ (see Subsection \ref{subsecunwe} for the definition of this term), so that finite extinction for solutions of the unweighted Leibenson equation (\ref{usualleibe}) holds in $M$.

The structure of the present paper is as follows.

In Section \ref{secweak} we define the notion of a weak solution of equation (\ref{evoeq}).

In Section \ref{Secini} we prove our main result about finite extinction time.

In Section \ref{examplessec} we construct more explicitly manifolds with finite extinction time. 
In Subsection \ref{subsecunwe} we construct and discuss the aforementioned weighted model with finite extinction time for solutions of the unweighted Leibenson equation (\ref{usualleibe}).
In Subsections \ref{secCH} and \ref{subsecNR} we give more detailed descriptions on function $\rho$ so that finite extinction holds for solutions of equation (\ref{evoeq}) in the case when $M$ is a Cartan-Hadamard manifold and a manifold of \textit{non-negative Ricci curvature}, respectively, by using known results about weighted Sobolev inequalities on these manifolds.

For results about solutions of the unweighted Leibenson equation (\ref{usualleibe}) on Riemannian manifolds in other cases of $p, q$, we refer to \cite{dekkers2005finite, grigor2023finite, Grigoryan2024, Grigoryan2024a, surig2024sharp}.

\begin{acknowledgement}
	The author would like to thank Alexander Grigor'yan for many helpful discussions.
\end{acknowledgement}
	
	\section{Weak subsolutions}
	
	\label{secweak}
	
	We consider in what follows the following non-linear evolution
	equation on a Riemannian manifold $M$:%
	\begin{equation}
		\rho\partial _{t}u=\Delta _{p}u^{q},  \label{dtv}
	\end{equation} where $\rho$ is a non-negative measurable function on $M$.
	By a \textit{subsolution} of (\ref{dtv}) we mean a non-negative function $u$
	satisfying 
	\begin{equation}
		\rho\partial _{t}u\leq \Delta _{p}u^{q}  \label{subdtv}
	\end{equation}%
	in a certain weak sense as explained below.
	
	We assume throughout that 
	\begin{equation*}
		p>1\ \ \text{and}\ \ \ q>0.
	\end{equation*}%
	Set%
	\begin{equation}\label{D}
		D =1-q(p-1).
	\end{equation}%
		
	Given a positive smooth function $h$, let $\mu$ be the measure defined by $$d\mu=hd\textnormal{vol},$$ where $\textnormal{vol}$ denotes the Riemannian measure on $M$. The pair $(M, \mu)$ is called \textit{weighted manifold}.
	
\begin{definition}
	\normalfont
	Let $\Omega$ be an open subset of $M$. Assume that $\rho\in L_{loc}^{1}(\Omega)$. We say that a non-negative function $u=u(x, t)$ is a \textit{weak
		subsolution} of (\ref{dtv}) in $\Omega\times I$, if 
	\begin{equation}  \label{defvonsoluq}
		u\in C\left(I; L^{\zeta}(\Omega, \rho d\mu)\right)\quad \textnormal{and}\quad 
		u^{q}\in L_{loc}^{p}\left(I; W^{1, p}(\Omega)\right),
	\end{equation} where $\zeta\geq1$
	and (\ref{subdtv}) holds weakly in $\Omega\times I$, which means that for all $t_{1}, t_{2}\in I$ with $t_{1}<t_{2}$, and all non-negative \textit{test functions} 
	\begin{equation}  \label{defvontestsoluq}
		\psi\in W_{loc}^{1, \frac{\zeta}{\zeta-1}}\left(I;
		L^{\frac{\zeta}{\zeta-1}}(\Omega, \rho d\mu)\right)\cap L_{loc}^{p}\left(I; W_{0}^{1,
			p}(\Omega)\right),
	\end{equation}
	we have 
	\begin{equation}  \label{defvonweaksolq}
		\left[\int_{\Omega}{u\psi\rho d\mu}\right]_{t_{1}}^{t_{2}}+\int_{t_{1}}^{t_{2}}{%
			\int_{\Omega}{-u\partial_{t}\psi\rho+|\nabla u^{q}|^{p-2}\langle\nabla u^{q},
				\nabla \psi\rangle}d\mu dt}\leq 0.
	\end{equation}
\end{definition}
	
	\textit{Weak solutions} of (\ref{dtv}) are
	defined analogously. Existence and uniqueness results
	for the Cauchy problem with the above notion of weak solutions of (\ref{dtv}) were obtained in the
	euclidean case, for example, in \cite{andreucci1990new, ishige1996existence} (see also \cite{tedeev2007interface}).

		From now on let us assume that $M$ is geodesically complete.
		
		The next lemma can be proved similarly to Lemma 2.6 in \cite{grigor2023finite}.
	\begin{lemma}[Caccioppoli type inequality]
		\label{Lem1}\label{LemIni} Let $I$ be an interval in $\mathbb{R}_{+}=[0, \infty)$ and let $u=u\left( x,t\right) $ be a bounded
		non-negative subsolution to \emph{(\ref{dtv})} in $M\times I$. Fix some real $\sigma $
		such that 
		\begin{equation}
			\sigma \geq \max \left(pq, \zeta-D\right).  \label{la>3-m}
		\end{equation}%
		Choose $t_{1},t_{2}\in I$ such that $t_{1}<t_{2}$. Then%
		\begin{equation}
			\left[ \int_{M }u^{\sigma+D }\rho d\mu\right] _{t_{1}}^{t_{2}}+c_{1}%
			\int_{M\times [t_{1}, t_{2}]}\left\vert \nabla \left( u^{\sigma/p } \right) \right\vert
			^{p}d\mu dt\leq 0,
			\label{vetacor}
		\end{equation}
		where $c_{1}$ is a positive constant depending on $p$, $q$, $\sigma $. 
	\end{lemma}

Let us observe for a later usage that \begin{equation}\label{valpha}u^{\sigma/p}\in L_{loc}^{p}\left(I; W^{1, p}(M)\right).\end{equation}
Indeed, using $\sigma/p\geq q$, we get that the function $\Phi(s)=s^{\frac{\sigma}{pq}}$ is Lipschitz on any bounded interval in $[0, \infty)$. Thus, $u^{\sigma/p}=\Phi(u^{q})\in W^{1, p}(M)$ and $$\left|\nabla u^{\sigma/p}\right|=\left|\Phi^{\prime}(u^{q})\nabla u^{q}\right|\leq C\left|\nabla u^{q}\right|,$$ whence for any bounded interval $J\subset I$, \begin{equation*}\label{valpha0}\int_{M\times J}u^{\sigma}+\left\vert \nabla \left( u^{\sigma/p } \right) \right\vert^{p}\leq  C^{\prime}\int_{M\times J}u^{\sigma}+\left\vert \nabla  u^{q } \right\vert^{p},\end{equation*} which is finite since $$\int_{M\times J}u^{\sigma}\leq \textnormal{const}~||u||_{L^{\infty}}^{\sigma-pq}\int_{M\times J}u^{pq}$$ and proves (\ref{valpha}).
	
\section{Finite extinction time}
	
\label{Secini} 
From now on we always assume that $D>0$ (see (\ref{D})).
\begin{definition}
We say that a function $u=u(x, t)$ has \textit{finite extinction time} in $M$ if there exists $T>0$ such that $$u(\cdot, t)\equiv0\quad \text{for all}~t\geq T.$$	
\end{definition}

\begin{definition}
Let $\omega$ be a positive measurable function in $M$. We say that the weighted manifold $(M, \mu)$ satisfies a weighted Sobolev inequality with weight $\omega$ if, for all $v\in W^{1, p}(M)$, there exists a constant $C>0$ such that \begin{equation}\label{weightsobo}\left(\int_{M}{|v|^{p\kappa}\omega d\mu}\right)^{1/\kappa}\leq C\int_{M}{|\nabla v|^{p}d\mu},\end{equation} where $\kappa>1$.
\end{definition}

The following theorem is our main result. It contains Theorem \ref{thmfinexint} from Introduction.

\begin{theorem}\label{thmfinex}
Let $u$ be a non-negative bounded subsolution of (\ref{dtv}) in $M\times \mathbb{R}_{+}$ with $u(\cdot, 0)=u_{0}\in L^{\zeta}(M, \rho d\mu)\cap L^{\infty}(M)$. Suppose that $M$ satisfies the weighted Sobolev inequality (\ref{weightsobo}) with weight $\omega$. Assume that \begin{equation}\label{finitecond}\left|\left|\frac{\rho}{\omega}\right|\right|_{L^{\theta}(M, \omega d\mu)}<\infty,\end{equation} for some $\theta=\theta(p, q, \kappa, \zeta)$. Then $u$ has finite extinction time in $M$.
\end{theorem}

\begin{remark}\label{optimalrange}
We will see in the proof that \begin{equation}\label{deftheta}\theta=\frac{\kappa}{ \kappa-1-\frac{D}{\sigma}},\end{equation} where \begin{equation}\label{sigmalarge}\sigma= \max\left(pq, \zeta-D, \frac{D}{\kappa-1}\right)\end{equation}

In particular, in the case when $\sigma=\frac{D}{\kappa-1}$, we have $\theta=\infty$. Hence, in this case, condition (\ref{finitecond}) becomes $$\left|\left|\frac{\rho}{\omega}\right|\right|_{L^{\infty}(M)}<\infty.$$
\end{remark}

\begin{proof}
Let $\sigma$ be large enough as in (\ref{sigmalarge}) and set $$\Phi(t)=\int_{M}{u^{\sigma+D}(\cdot, t)\rho ~d\mu}$$ and set also $v=u^{\sigma/p}$. Note that by (\ref{sigmalarge}), we always have $\frac{\sigma \kappa}{\sigma+D}\geq1$. Let us first assume that $\frac{\sigma \kappa}{\sigma+D}>1$. Then we obtain from Hölder's inequality with the conjugated Hölder exponents $$\frac{\sigma \kappa}{\sigma+D}>1\quad\textnormal{and}\quad \frac{\sigma \kappa}{\sigma\kappa-(\sigma+D)}=:\theta<\infty,$$ \begin{align*}
\Phi&=\int_{M}{u^{\sigma+D}\rho d\mu}=\int_{M}{v^{p\frac{\sigma+D}{\sigma}}\rho d\mu}=\int_{M}{v^{p\frac{\sigma+D}{\sigma}}\omega^{\frac{\sigma+D}{\sigma \kappa}}\rho \omega^{-\frac{\sigma+D}{\sigma \kappa}}d\mu}\\&\leq \left(\int_{M}{v^{p\kappa}\omega d\mu}\right)^{\frac{\sigma +D}{\sigma \kappa}}\left(\int_{M}\left(\frac{\rho^{\sigma \kappa}}{\omega^{\sigma+D}}\right)^{\frac{1}{\sigma \kappa-(\sigma+D)}}d\mu\right)^{\frac{\sigma\kappa-(\sigma+D)}{\sigma \kappa}}\\&=\left(\int_{M}{v^{p\kappa}\omega d\mu}\right)^{\frac{\sigma+D}{\sigma \kappa}}\left(\int_{M}\left(\frac{\rho}{\omega}\right)^{\theta}\omega d\mu\right)^{\frac{1}{\theta}},\end{align*} where by (\ref{finitecond}), $$\left(\int_{M}\left(\frac{\rho}{\omega}\right)^{\theta}\omega d\mu\right)^{\frac{1}{\theta}}<\infty.$$

In the case $\frac{\sigma \kappa}{\sigma+D}=1$ we have by the same arguments $$\Phi\leq \left(\int_{M}{v^{p\kappa}\omega d\mu}\right)^{\frac{1}{\kappa}}\left|\left|\frac{\rho}{\omega}\right|\right|_{L^{\infty}(M)},$$ so that in both cases \begin{equation}\label{hoelderandfinite}\Phi\leq\left(\int_{M}{v^{p\kappa}\omega d\mu}\right)^{\frac{\sigma+D}{\sigma \kappa}}\left|\left|\frac{\rho}{\omega}\right|\right|_{L^{\theta}(M, \omega d\mu)}.\end{equation}

Recall that by (\ref{valpha}), $ v(\cdot, t)\in W^{1, p}(M)$ for all $t>0$. Hence, combining (\ref{hoelderandfinite}) and (\ref{weightsobo}) we deduce \begin{equation}\label{upperphi}\Phi\leq C \left(\int_{M}{|\nabla v|^{p}}d\mu\right)^{\frac{\sigma+D}{\sigma}}.\end{equation} From (\ref{vetacor}) we have for $t_{1}<t_{2}$, $$\Phi(t_{2})-\Phi(t_{1})\leq -c_{1}
\int_{M\times[t_{1}, t_{2}]}\left\vert \nabla v \right\vert
^{p}d\mu dt,$$ so that, in particular, $\Phi$ is monotone decreasing. Therefore, $\Phi(t)$ is differentiable for almost every $t>0$ and $$\frac{d}{dt}\Phi\leq -c_{1} \int_{M}\left\vert \nabla v \right\vert
^{p}d\mu.$$ Combining this with (\ref{upperphi}), we get  $$C^{-\frac{\sigma}{\sigma+D}}\Phi^{\frac{\sigma}{\sigma+D}}\leq \int_{M}{|\nabla v|^{p}}\leq -c_{1}^{-1}\frac{d}{dt}\Phi,$$ so that $$\frac{d}{dt}\Phi\leq -c \Phi^{\frac{\sigma}{\sigma+D}},$$ where $c=c_{1}C^{-\frac{\sigma}{\sigma+D}}$. Note that, since $\sigma+D\geq \zeta$, we have $u\in C([0, +\infty), L^{\zeta}(M, \rho d\mu))$. Hence, we can solve the following ODE: $$\frac{d}{dt}\Psi=-c\Psi^{\frac{\sigma}{\sigma+D}},\quad \Psi(0)=\Phi(0),$$ and obtain $$\Psi(t)=\left(\Phi(0)^{\frac{D}{\sigma+D}}-\frac{D}{\sigma+D}ct\right)^{\frac{D}{\sigma+D}},\quad \textnormal{for}\quad0\leq t\leq \frac{\sigma+D}{cD}\Phi(0)^{\frac{D}{\sigma+D}}=:T.$$ By comparison, this implies that $\Phi\leq \Psi$ on $[0, T]$ and $\Phi=0$ for $t\geq T$, which finishes the proof.
\end{proof}

\section{Examples}\label{examplessec}

\subsection{Unweighted Leibenson equation}\label{subsecunwe}

Here we want to construct a class of weighted manifold so that any non-negative bounded subsolution of the unweighted Leibenson equation $$\partial_{t}u= \Delta _{p}u^{q}$$ has finite extinction time in $M$.

Let $(M, g, \mu)$ be a weighted manifold. Let $a$ and $b$ be positive smooth functions on $M$ and consider a new weighted manifold $(M, \widetilde{\mu}, \widetilde{g})$ with metric and measure on $M$: \begin{equation}\label{changemetricmeasureex}\widetilde{g}=ag\quad\text{and}\quad d\widetilde{\mu}=b d\mu.\end{equation}

\begin{lemma}\label{lemweightsobo}
Suppose that the manifold $(M, g, \mu)$ satisfies the Sobolev inequality \begin{equation}\label{sobolevohneweight}\left(\int_{M}{|v|^{p\kappa} d\mu}\right)^{1/\kappa}\leq C\int_{M}{|\nabla v|^{p}d\mu}.\end{equation}	
Consider the manifold $(M, \widetilde{\mu}, \widetilde{g})$ with metric and measure given by (\ref{changemetricmeasureex}) and assume that \begin{equation*}a^{p/2}=b.\end{equation*} Then $(M, \widetilde{\mu}, \widetilde{g})$ satisfies the weighted Sobolev inequality (\ref{weightsobo}) with weight \begin{equation}\label{defomegaforchange}\omega=\frac{1}{a^{p/2}}.\end{equation}
\end{lemma}

\begin{proof}
Because $\widetilde{\nabla}=\frac{1}{a}\nabla$, we have $$|\widetilde{\nabla} v|_{\widetilde{g}}^{2}=\langle \widetilde{\nabla}v, \widetilde{\nabla}v\rangle_{\widetilde{g}}=\frac{1}{a}\langle \nabla v, \nabla v\rangle=\frac{1}{a}|\nabla v|^{2}.$$ Therefore, $$\int_{M}|\widetilde{\nabla}v|_{\widetilde{g}}^{p}d\widetilde{\mu}=\int_{M}\frac{1}{a^{p/2}}|\nabla v|^{p}a^{p/2} d\mu=\int_{M}|\nabla v|^{p}d\mu,$$ which implies, by using (\ref{sobolevohneweight}), $$\left(\int_{M}{|v|^{p\kappa}\omega d\widetilde{\mu}}\right)^{1/\kappa}=\left(\int_{M}{|v|^{p\kappa} d\mu}\right)^{1/\kappa}\leq C\int_{M}{|\nabla v|^{p}d\mu}= \int_{M}|\widetilde{\nabla}v|_{\widetilde{g}}^{p}d\widetilde{\mu},$$ so that $(M, \widetilde{\mu}, \widetilde{g})$ satisfies the weighted Sobolev inequality (\ref{weightsobo}) with weight $\omega$ given by (\ref{defomegaforchange}).
\end{proof}

\begin{example}\label{examplesobeuc}
Assume that $n>p$ and let $a$ be a positive smooth function and let us consider the manifold $(M, g, \mu)$ with metric and measure given by \begin{equation}\label{exampleRn}g=a g_{eucl}\quad\text{and}\quad d\mu=a^{p/2}dx,\end{equation} where $dx$ denotes the Lebesgue measure on $(\mathbb{R}^{n}, g_{eucl})$. Since the manifold $(\mathbb{R}^{n}, g_{eucl}, dx)$ satisfies the Sobolev inequality (\ref{sobolevohneweight}) with \begin{equation}
	\kappa =\dfrac{n}{n-p},\label{k}
\end{equation} we obtain from Lemma \ref{lemweightsobo} that the manifold $(M, g, \mu)$ satisfies the weighted Sobolev inequality (\ref{weightsobo}) with weight $\omega$ given by (\ref{defomegaforchange}).
\end{example}	

Let us consider the topological space $M=(0, +\infty)\times \mathbb{S}^{n-1}$, $n\geq 2$, so that any point $x\in M$ can be written in the polar form $x=(r, \varphi)$ with $r\in(0, +\infty)$ and $\varphi \in \mathbb{S}^{n-1}$. We equip $M$ with the Riemannian metric $ds^{2}$ that is defined in polar coordinates $(r, \varphi)$ by $$ds^{2}=dr^{2}+\psi^{2}(r)d\varphi^{2}$$ with $\psi(r)$ being a smooth positive function on $(0, +\infty)$ and $d\varphi^{2}$ being the Riemannian metric on $\mathbb{S}^{n-1}$. In this case, $M$ is called a \textit{Riemannian model}.
The Riemannian measure $\textnormal{vol}$ on $M$ with respect to this metric is given by \begin{equation*}\label{measonmodpsi}d\textnormal{vol}=\psi^{n-1}(r)drd\nu(\varphi),\end{equation*} where $dr$ denotes the Lebesgue measure on $(0, +\infty)$ and $d\nu$ denotes the Riemannian measure on $\mathbb{S}^{n-1}$. 
Then the \textit{area function} $S$ on $(0, +\infty)$ is given by \begin{equation*}\label{defofareabd}S(r)=\nu(\mathbb{S}^{n-1})\psi^{n-1}(r).\end{equation*}

Given a smooth positive function $h$ on $M$, that only depends on the polar radius $r$, let us define a measure $\mu$ on $M$ by $$d\mu=hd\textnormal{vol}.$$ We obtain that the weighted manifold $(M, \mu)$ has the area function \begin{equation*}\label{areafuncontild}S_{\mu}(r)=h(r)S(r).\end{equation*} Then the weighted manifold $(M, \mu)$ is called a \textit{weighted model} and we get that \begin{equation}\label{fubinionmode}d\mu=S_{\mu}(r)drd\nu(\varphi).\end{equation}

\begin{lemma}\label{weightedmodelrho}
Let $(M, g, \mu)$ be the weighted manifold with metric and measure given by (\ref{exampleRn}), that is, \begin{equation*}g=a g_{eucl}\quad\text{and}\quad d\mu=a^{p/2}dx,\end{equation*} where $dx$ denotes the Lebesgue measure on $(\mathbb{R}^{n}, g_{eucl})$.
Denote with $(R, \varphi)$ the euclidean coordinates and suppose that for some constant $c>0$ and all large $R$, $$a(x)=a(R)=c R^{-l},$$ where $l\in [0, 2)$. Then the manifold $(M, g)$ is a model manifold and the weighted model $(M, g, \mu)$ has an area function $S_{\mu}$ such that, for large $r$, \begin{equation}\label{areweightmod}S_{\mu}(r)\simeq r^{\frac{l(n-p)}{2-l}+n-1},\end{equation} where $r$ denotes the Riemannian polar radius in $(M, g)$.
\end{lemma}

\begin{proof}
We have \begin{equation}\label{metricunweight}g=ag_{eucl}=a\left(dR^{2}+R^{2}d\varphi^{2}\right)=dr^{2}+\left(1-\frac{l}{2}\right)^{2}r^{2}d\varphi^{2},\end{equation} where $$r=\int_{0}^{R}\sqrt{a}=r_{0}+\sqrt c\int_{1}^{R}s^{-l/2}ds=\frac{\sqrt c}{1-l/2}R^{1-l/2}$$ for large $R$ and $r_{0}$ is chosen such that $r_{0}=\frac{\sqrt c}{1-l/2}$. 
For the Lebesgue measure $dx_{g}$ on $(M, g)$ we have $$dx_{g}=a^{n/2}dx.$$ Hence, we obtain \begin{equation}\label{weightedmeaex}d\mu=a^{p/2}dx=a^{-\frac{n-p}{2}}dx_{g}.\end{equation} This yields the estimate  $$d\mu\simeq R^{\frac{l(n-p)}{2}}dx_{g}\simeq r^{\frac{l(n-p)}{2-l}}dx_{g}\simeq r^{\frac{l(n-p)}{2-l}+n-1}dr,$$ so that the weighted model $(M, g, \mu)$ has an area function $S_{\mu}$ given by $$S_{\mu}(r)\simeq r^{\frac{l(n-p)}{2-l}+n-1},$$ which is (\ref{areweightmod}).
\end{proof}

\begin{remark}\label{remarkweight}
It follows from (\ref{metricunweight}) that the Riemannian model $(M, g, \textnormal{vol})$ has an area function \begin{equation}\label{arearieman}S(r)=\nu(\mathbb{S}^{n-1})\left(1-\frac{l}{2}\right)^{n-1} r^{n-1}.\end{equation} Thus, (\ref{areweightmod}) implies that the manifold $(M, g, \mu)$ is a weighted model with weight $$h(r)\simeq r^{\frac{l(n-p)}{2-l}}.$$
In particular, in the case $l=0$, we have $(M, g, d\mu)=(\mathbb{R}^{n}, g_{eucl}, dx)$.
\end{remark}

\begin{proposition}\label{weightmodrhoone}
Assume that $n>p$. Let $(M, g, \mu)$ be the weighted model from Lemma \ref{weightedmodelrho} such that there exists some constant $c>0$ so that all for large $R$, $$a(x)=a(R)=c R^{-l},$$ where \begin{equation}\label{rangeoflrho1}l\in\left\{ 
	\begin{array}{ll}
		[0, 2), & \text{if }\theta=\infty, \\ 
		\left(2\frac{n}{p\theta }, 2\right), & \text{if }\theta<\infty,\\%
	\end{array}%
	\right.\end{equation} and $\theta$ is as in (\ref{deftheta}).
Then any non-negative bounded subsolution of $$\partial_{t}u= \Delta _{p}u^{q}\quad \text{in}~M\times \mathbb{R}_{+}$$ has finite extinction time in $M$.
\end{proposition}	

\begin{remark}
Let $\kappa$ be defined as in (\ref{k}). Then
\begin{equation*}
	\frac{\kappa }{\kappa -1}=\dfrac{n}{p}
\end{equation*}	
and therefore,  $$\frac{n}{p\theta}<\frac{n}{p}\frac{\kappa-1}{\kappa}= 1,$$ so that the intervals in (\ref{rangeoflrho1}) are non-empty. 

In particular, we have $$\theta=\left\{ 
\begin{array}{ll}
	\dfrac{n}{p-\frac{D(n-p)}{pq}}, & \text{if }p>\frac{D(n-p)}{pq}~\text{and}~ \zeta\leq q+1, \\ 
	\dfrac{n}{p-\frac{D(n-p)}{\zeta-D}}& \text{if}~\zeta\geq \max\left(q+1, \frac{Dn}{p}\right),\\\infty, &\text{else},%
\end{array}%
\right. .$$
and therefore, $$\frac{n}{p\theta }=\left\{ 
\begin{array}{ll}
\dfrac{p-\frac{D(n-p)}{pq}}{p}, & \text{if }p>\frac{D(n-p)}{pq}~\text{and}~ \zeta\leq q+1 \\ \dfrac{p-\frac{D(n-p)}{\zeta-D}}{p}& \text{if}~\zeta\geq \max\left(q+1, \frac{Dn}{p}\right),\\0, &\text{else}.
\end{array}%
\right. $$

For example, we have $\theta=\infty$ if $n$ is large enough.
\end{remark}

\begin{proof} From Lemma \ref{weightedmodelrho} we know that the manifold $(M, g)$ is geodesically complete. It follows from Example \ref{examplesobeuc} that the weighted model $(M, g, \mu)$ satisfies the weighted Sobolev inequality (\ref{weightsobo}) with weight $\omega$ given by $$\omega=\frac{1}{a^{p/2}}.$$ Hence, in order to apply Theorem \ref{thmfinex}, it remains to show that condition (\ref{finitecond}) is satisfied with $\rho\equiv1$, that is, \begin{equation}\label{condrho1}\left|\left|\frac{1}{\omega}\right|\right|_{L^{\theta}(M, \omega d\mu)}<\infty.\end{equation} In the case $\theta<\infty$, we have $$\int_{M}\omega^{1-\theta}d\mu=\int_{\mathbb{R}^{n}}\omega(x)^{-\theta}dx=\int_{\mathbb{R}^{n}}a(x)^{\frac{p}{2}\theta}dx\simeq \int_{R_{0}}^{\infty} R^{-\frac{lp}{2}\theta+n-1}dR.$$ Noticing that for $l$ in (\ref{rangeoflrho1}) we have $$n<\frac{lp}{2}\theta,$$ we deduce (\ref{condrho1}) in this case. On the other hand, for $\theta=\infty$, we see that for large $R$, $$\frac{1}{\omega}=a^{p/2}\simeq R^{-\frac{lp}{2}}$$ is bounded for $l\geq 0$. Therefore, we can apply Theorem \ref{thmfinex}, which yields the claim.	
\end{proof}

\subsection{Cartan-Hadamard manifolds}\label{secCH}

Let $x_{0}\in M$ be some fixed point and denote $V(r)=\textnormal{vol}(B(x_{0}, r))$ and $|x|=d(x, x_{0})$.

\begin{theorem}\label{mainthmCH}
Let $M$ be a Cartan-Hadamard manifold. Let $\theta$ be as in (\ref{deftheta}) and suppose that $\rho$ satisfies \begin{equation}\label{condCHm}\left|\left|\rho\right|\right|_{L^{\theta}(M, d\text{vol})}<\infty.\end{equation} Assume that $n>p$. Then any non-negative bounded subsolution of $$\rho\partial_{t}u= \Delta _{p}u^{q}\quad \text{in}~M\times \mathbb{R}_{+}$$ has finite extinction time in $M$.
\end{theorem}

\begin{proof}
It follows from \cite{hoffman1974sobolev} that on $(M, \textnormal{vol})$ the weighted Sobolev inequality (\ref{weightsobo}) holds with weight $\omega\equiv 1$ and $\kappa=\frac{n}{n-p}$. Hence, the claim follows from Theorem \ref{thmfinex}.
\end{proof}

\begin{example}
Let $M$ be a Cartan-Hadamard manifold and suppose that there exist $\alpha>0$ so that for all large enough $r$, \begin{equation}\label{volch}V(r)\simeq r^{\alpha}.\end{equation} Suppose also that $\rho(x)=\rho(|x|)\simeq |x|^{-l}$ for large $|x|$, where $l\geq0$ is to be chosen. Suppose first that $\theta<\infty$. Setting $S_{r}=\{x\in M:|x|=r\}$ and denoting by $\nu$ the Riemannian measure on $S_{r}$ we have $$\int_{M}\rho^{\theta}d\textnormal{vol}\simeq \int_{r_{0}}^{\infty}\int_{S_{r}}\rho(r)^{\theta}d\nu dV(r)\simeq \int_{r_{0}}^{\infty}r^{-l\theta}dV(r) $$ so that condition (\ref{condCHm}) becomes $$\int_{r_{0}}^{\infty}r^{-l\theta}dV(r)\simeq\left[r^{\alpha-l\theta}\right]_{r_{0}}^{\infty}+\int_{r_{0}}^{\infty}r^{\alpha-l\theta-1}dr<\infty.$$ Therefore, condition (\ref{condCHm}) is implied if \begin{equation}\label{CHml}l>\frac{\alpha}{\theta}.\end{equation} In the case $\theta=\infty$, $\rho$ is bounded for all $l\geq 0$, so that (\ref{condCHm}) also holds in this case.
\end{example}

\begin{remark}\normalfont
When $M=\mathbb{R}^{n}$ we have (\ref{volch}) with $\alpha=n$ so that (\ref{CHml}) becomes $l>\frac{n}{\theta}$.
It is shown in \cite{tedeev2007interface} that the optimal range for $l$ in $\mathbb{R}^{n}$ is
\begin{equation}\label{optrangeRn}
	l \left\{ 
	\begin{array}{ll}
		>\dfrac{p-nD}{1-D}, & \text{if }p\geq nD, \\ 
		\geq 0, & \text{if }p< nD,%
	\end{array}%
	\right.  
\end{equation}%
which matches (\ref{CHml}) with the optimal $\theta$ given by (\ref{thetamax}) if $p\geq nD$. Hence, as described in Remark \ref{remarkweight}, it follows from (\ref{sigmalarge}) that our $\theta<\frac{n(1-D)}{p-nD}$ so that in this case our range for $l$ seems to be non-optimal. On the other hand, in the case when $n$ is large enough (in particular, so that $p< nD$) we get the optimal range $l\geq 0$.
\end{remark}

\subsection{Manifolds of non-negative Ricci curvature}\label{subsecNR}

\begin{proposition}\label{nonnegricci}
Let $(M, g, \textnormal{vol})$ be a manifold with non-negative Ricci curvature and with reverse volume doubling, where the latter means that there exists $\alpha>2$ and $C>0$ such that for all $R\geq r>0$, $$\frac{V(R)}{V(r)}\geq C\left(\frac{R}{r}\right)^{\alpha}.$$ Assume that $n>p$. Suppose that $\rho$ satisfies \begin{equation}\label{condnonricci}\left|\left|\frac{\rho(x)|x|^{\kappa p}}{V(|x|)^{\kappa-1}}\right|\right|_{L^{\theta}\left(M,  \frac{V(|x|)^{\kappa-1}}{|x|^{\kappa p}}d\text{vol}\right)}<\infty,\end{equation} where $\theta$ is as in (\ref{deftheta}) and $\kappa=\frac{n}{n-p}$. Then any non-negative bounded subsolution of $$\rho\partial_{t}u= \Delta _{p}u^{q}\quad \text{in}~M\times \mathbb{R}_{+}$$ has finite extinction time in $M$.
\end{proposition}

\begin{proof}
It follows from Theorem 0.1 from \cite{minerbe2009weighted} that on this manifold, the weighted Sobolev inequality (\ref{weightsobo}) holds with weight $$\omega(x)=\omega(|x|)=\frac{V(|x|)^{\kappa-1}}{|x|^{\kappa p}}.$$ Hence, the claim follows again from Theorem \ref{thmfinex}.
\end{proof}

\begin{example}
Let $(M, g, \mu)$ be as in Proposition \ref{nonnegricci} and suppose also that there exists $\alpha>0$ so that for all large enough $r$, $$V(r)\simeq r^{\alpha}.$$ Assume also that $\rho(x)=\rho(|x|)\simeq |x|^{-l},$ where $l\geq 0$ is to be chosen. In the case $\theta<\infty$, we have \begin{align*}\int_{M}\rho^{\theta}\left(\frac{r^{ \kappa p}}{V(r)^{\kappa-1}}\right)^{\theta-1}d\textnormal{vol}&\simeq\int_{r_{0}}^{\infty} r^{-l\theta+(\kappa p-\alpha(\kappa-1))(\theta-1)}dV(r)\\&\simeq\left[r^{\alpha-l\theta+(\kappa p-\alpha(\kappa-1))(\theta-1)}\right]_{r_{0}}^{\infty}\\&\quad+\int_{r_{0}}^{\infty}r^{\alpha-l\theta+(\kappa p-\alpha(\kappa-1))(\theta-1)-1}dr.\end{align*} Thus, if  \begin{equation}\label{rangeleuc}l>\frac{\alpha+(\kappa p-\alpha(\kappa-1))(\theta-1)}{\theta}=\alpha-\frac{\kappa(\theta-1)(\alpha-p)}{\theta},\end{equation}the condition (\ref{condnonricci}) is satisfied in this case. On the other hand, we have for large $|x|$, $$\frac{\rho(x)|x|^{\kappa p}}{V(|x|)^{\kappa-1}}\simeq |x|^{-l+\kappa p-\alpha(\kappa-1)},$$ which is bounded if $l\geq \kappa p-\alpha(\kappa-1)$. 
\end{example}

\begin{remark}
In the case when $M=\mathbb{R}^{n}$ we have $\alpha=n$. Then the range (\ref{rangeleuc}) matches the optimal range (\ref{optrangeRn}) when $\theta$ is given by (\ref{thetamax}) if $p\geq nD$. Hence, as described in Remark \ref{optimalrange} our $\theta<\frac{n(1-D)}{p-nD}$, so that in this case our range for $l$ again seems to be non-optimal. On the other hand, we get the optimal range $l\geq0$ in the case when $n$ is large enough (in particular, so that $p< nD$).
\end{remark}

\begin{da} \normalfont
	This article has no associated data.
\end{da}

\bibliographystyle{abbrv}
\bibliography{librarycacc}

\emph{Universit\"{a}t Bielefeld, Fakult\"{a}t f\"{u}r Mathematik, Postfach
	100131, D-33501, Bielefeld, Germany}

\texttt{philipp.suerig@uni-bielefeld.de}
\end{document}